\documentclass[a4paper]{amsart}

\RequirePackage[nonfrench]{kotex}
\usepackage[utf8]{inputenc}
\usepackage{kotex}
\usepackage{epsfig}

\usepackage{amsfonts}
\usepackage{amssymb}
\usepackage{amsmath}
\usepackage{amsthm}
\usepackage{amscd}
\usepackage{amstext}
\usepackage{latexsym}
\usepackage{array}
\usepackage{graphicx, subcaption}
\usepackage{todonotes}
\usepackage{ulem}

\usepackage{mathrsfs}
\usepackage{multicol}
\usepackage{graphicx}
\usepackage{calc}
\usepackage{cancel}
\newcommand\Ccancel[2][black]{
    \let\OldcancelColor\CancelColor
    \renewcommand\CancelColor{\color{#1}}
    \cancel{#2}
    \renewcommand\CancelColor{\OldcancelColor}
}
\usepackage{hyperref}

\usepackage{enumitem}

\usepackage{fancyhdr}
\fancypagestyle{plain}{
\fancyhf{}
\fancyhead[R]{\thepage}

}
\usepackage{tikz}
\usetikzlibrary{hobby}
\usetikzlibrary{arrows}
\usetikzlibrary{cd}

\usepackage{amsrefs}
\newcommand{\CC}{\mathbf{C}}

\newcommand{\RR}{\mathbf{R}}

\newcommand{\mfa}{\mathfrak{a}}

\newcommand{\p}[2][]{\partial{#1} / \partial{#2}}
\newcommand{\pp}[2][]{\frac{\partial{#1}}{\partial{#2}}}

\providecommand{\pf}{\mathrm{pf}}

\theoremstyle{plain}
\newtheorem{theorem}{Theorem}[section]
\newtheorem{lemma}[theorem]{Lemma}
\newtheorem{corollary}[theorem]{Corollary}
\newtheorem{proposition}[theorem]{Proposition}

\theoremstyle{definition}
\newtheorem{definition}[theorem]{Definition}

\newtheorem{example}[theorem]{\textnormal{\textbf{Example}}}

\newtheorem{Standard Process}[theorem]{Procedures}

\theoremstyle{remark}
\newtheorem{remark}[theorem]{Remark}

\begin{document}

\title{On contact Hamiltonian functions of singular contact structures}



\author{Hoseob Seo}

\begin{abstract}
For contact manifolds, it is well-known that the map which assigns to an infinitesimal contact transformation its contact Hamiltonian function is a linear isomorphism, and an explicit local formula for its inverse can be given. In contrast, infinitesimal contact transformations on a singular contact structure have not been well-studied yet. In this article, we show that this map is injective and present an explicit local formula for its inverse when the contact structure has singularities of the first type with structurally smooth Martinet hypersurface.
\end{abstract}

\maketitle

\section{Introduction}

Let $M$ be a complex manifold of complex dimension $2k+1$ ($k \ge 1$). A (holomorphic) distribution $\mathcal{D}$ of dimension $2k$ on $M$ is called a contact structure if it satisfies the contact condition. In this case, one can find a family of local contact $1$-forms, $\omega = \{\omega_i\}$ whose kernel is $\mathcal{D}$. Moreover, such a contact form defines the globally defined $1$-form $\widetilde{\omega}$ with values in the quotient bundle $L:=T^{1,0}M/\mathcal{D}$ on $M$. Denote by $\mfa(M,\widetilde{\omega})$ the Lie algebra of infinitesimal contact transformations of $M$, that is, the set of holomorphic vector fields $X$ on $M$ such that $\mathcal{L}_X \widetilde{\omega}$ locally equals to $h\widetilde{\omega}$ for some holomorphic function $h$. The contact $L$-valued $1$-form $\widetilde{\omega}$ defines the map $\theta: \mfa(M,\widetilde{\omega}) \longrightarrow H^0(M,L)$ by $\theta(X) = \widetilde{\omega}(X)$, where $\theta (X)$ is called the contact Hamiltonian section of $L$ associated to $X$. This map is indeed a linear isomorphism. Additionally, for a given section $s \in H^0(M, L)$, a local formula for the vector field $X$ whose image under $\theta$ is $s$ is also known (cf. \cite{Ko}, \cite{Ly75}).

In this article, we mainly study singular contact structures of the first type and their infinitesimal contact transformations. By a singular contact structure of the first type, we mean a family of locally defined nowhere vanishing $1$-forms $\omega = \{ \omega_i \}$ such that $\omega \wedge (d\omega)^k$ vanishes along a subvariety of codimension $1$ in $M$. In this setting, one can also consider the map 
$$
\theta : \mfa(M,\widetilde{\omega}) \longrightarrow H^0(M, L)
$$
defined as before. In contrast to nonsingular contact structures, this map may not be an isomorphism. However, we can say that this map is injective for a certain class of singular contact forms.
\begin{theorem}[=Theorem~\ref{THM:EvaluatingIsInjective}]
     Let $\omega = \{\omega_i\}$ be a singular contact structure on $M$ of the first type with structurally smooth Martinet hypersurface. Then the the sequence of sheaves
     $$
     0 \longrightarrow \mathfrak{a}(M,\widetilde \omega) \stackrel{\theta}{\longrightarrow} \mathcal{O}(L)
     $$
     is exact.
\end{theorem}

It is worth noting that we also present an explicit local criterion for a germ of section $s \in \mathcal{O}(L)_x$ to be the contact Hamiltonian section associated to an infinitesimal contact transformation $X$, and an explicit local formula of $X$ for a given contact Hamiltonian function. As abuse of notation, we will not distinguish between $\omega$ and $\widetilde{\omega}$ if there is no risk of confusion. 

As applications of these results, we establish that (1) if a section of $L$ which belongs to the image of $\theta$ vanishes along $S$, then the vanishing order of it must be at least $2$, and (2) every infinitesimal contact transformation of a singular contact structure preserves the Martinet hypersurface, that is, every infinitesimal contact transformation is holomorphically tangent to the Martinet hypersurface. These phenomena are naturally expected when we consider singular contact structures.

\begin{corollary}[=Corollary~\ref{CORO:VanishingOrder1ImpliesVanishingOrder2}]
Let $\omega$ be a singular contact structure on $M$ of the first type with structurally smooth Martinet hypersurface $S$. If $I_S$ is the ideal sheaf of $S$, then we have
$$
\theta(\mfa(M,\omega)) \cap H^0(M, L \otimes I_S) \subseteq H^0(M, L \otimes I_S^2).
$$
In particular, the infinitesimal contact transformation $X$ corresponding to $f \in \theta(\mfa(M,\omega)) \cap H^0(M, L \otimes I_S)$ vanishes everywhere on $S$.
\end{corollary}

\begin{theorem}[=Theorem~\ref{THM:PreserveMartinetHypersurface}]
Let $\omega$ be a singular contact structure on $M$ of the first type with structurally smooth Martinet hypersurface $S$. If a holomorphic vector field $X$ is an infinitesimal contact transformation of $\omega$, then the restriction $X|_S$ to the Martinet hypersurface $S$ is a holomorphic tangent vector field on $S$.
\end{theorem}

This paper is organized as follows. In Section~\ref{SEC:Preliminaries}, we recall preliminary things on nonsingular contact structures and the notion of Pfaffians of matrices. In Section~\ref{SEC:MainStatements}, we give the full statements of our main results. In Section~\ref{SEC:ProofOfLEM:InfAutoCondition} and Section~\ref{SEC:ProofOfTHM:InfAutoFromFtn}, we present the proofs of our main results. In Section~\ref{SEC:Applications}, we give a few elementary examples and several applications of our results, including the fact that the Martinet hypersurface is invariant under infinitesimal contact transformations. \\

{\noindent\bf Acknowledgement.} 
The author would like to thank Jun-Muk Hwang for suggesting the problem and providing valuable comments. The author was supported by the Institute for Basic Science (IBS-R032-D1-2025-a00) and by the National Research Foundation of Korea(NRF) grant funded by the Korea government(MSIT) (RS-2025-00556878).

\bigskip

\section{Preliminaries} \label{SEC:Preliminaries}

A (holomorphic) {\bf contact form} $\omega$ on $M$ is a collection of pairs $\{(U_i, \omega_i )\}_{i \in I}$, where $\{U_i\}_{i\in I}$ is an open cover of $M$ indexed by a set $I$ and each $\omega_i$ is a (holomorphic) $1$-form on $U_i$ satisfying
\begin{enumerate}
    \item $\omega_i \wedge (d\omega_i)^k$ is nowhere vanishing on $U_i$, and
    \item for any $i,j \in I$, there exists $f_{ij} \in \Gamma (U_i \cap U_j, \mathcal{O}_M^*)$ such that $\omega_i = f_{ij} \omega_j$.
\end{enumerate}
For two contact forms $\omega:= \{(U_i, \omega_i)\}_{i\in I}$ and $\mu := \{(V_j, \mu_j)\}_{j \in J}$, we say that they are {\bf equivalent} if for each $i\in I$ and $j \in J$, there exists a holomorphic function $h_{ij} \in \Gamma(U_i \cap V_j, \mathcal{O}_M^*)$ such that $\omega_i = h_{ij} \mu_j$ on $U_i \cap V_j$. A {\bf contact structure} on $M$ is an equivalence class of contact forms on $M$. Equivalently, a contact structure on $M$ can be given as a (holomorphic) subbundle $\mathcal{D}$ of the (holomorphic) tangent bundle $T^{1,0}M$ such that the map from $\bigwedge^2 \mathcal{D} to T^{1,0}M/\mathcal{D}$ which assigns the equivalence class of $[\eta, \xi]$ in $T^{1,0}M/\mathcal{D}$ to $\eta \wedge \xi$ is nondegenerate on $M$. For convenience, write $L:=T^{1,0}M/\mathcal{D}$.

Since $\omega_i$ defines a nontrivial linear functional on $L_x$ for each $x \in U_i$, we can find a section $s_i$ of $L$ on $U_i$ such that $\omega_i (s_i) \equiv 1$. Then $\widetilde{\omega} := s_i \otimes \omega_i$ is a globally defined $L$-valued $1$-form on $M$. Note that $\widetilde{\omega}$ depends only on the equivalence class of the contact structure $\omega$.

\begin{definition}
A (holomorphic) vector field $X$ on $M$ is called an {\bf infinitesimal contact transformation} of the contact structure $\omega$ if, for each $i \in I$, $\mathcal{L}_X \omega_i = h_i \omega_i$ for some $h_i \in \Gamma(U_i, \mathcal{O}_M)$. Denote by $\mfa (M, \widetilde{\omega})$ the set of infinitesimal contact transformations of $\omega$.
\end{definition}

For a infinitesimal contact transformation $X \in \mfa (M, \widetilde{\omega})$, the contact $1$-form $\widetilde{\omega}$ naturally gives a section of $L$ by $\widetilde{\omega}(X)$. Denote by $\theta$ this linear map from $\mfa(M,\widetilde{\omega})$ to $H^0 (M, L)$. As mentioned in the introduction, the map $\theta$ is a linear isomorphism, and an explicit formula for the preimage of a given $s \in H^0(M,L)$ is well-known.

\begin{theorem}
The map $\theta : \mfa(M, \widetilde{\omega}) \longrightarrow H^0(M,L)$ is a linear isomorphism. Furthermore, if $\omega_i (X) = f \in \Gamma(U_i, \mathcal{O}_M)$ and $U_i$ is sufficiently small so that
$$
\omega_i = dz + x^1 dx^{k+1} + x^2 dx^{k+2} + \cdots + x^k dx^{2k}
$$
in local coordinates $(z, x^1,\ldots, x^{2k})$, then $X$ is given by
$$
X =  \left( f - \sum_{i=1}^{k} x^i \pp[f]{x^i} \right) \pp{z} + \sum_{i=1}^{k}\left\{ \left(x^i\pp[f]{z} - \pp[f]{x^{k+i}}\right) \pp{x^i} + \pp[f]{x^i} \pp{x^{k+i}} \right\}
$$
on $U_i$.
\end{theorem}

As an application of the above formula, it can be shown that the group of contact transformations with compact support acts $k$-fold transitively on $M$ (See \cite{Boo69}, \cite{Hat66}, and \cite{Ko} for details).

On a singular contact structure, the map $\theta$ fails to be an isomorphism since $\omega \wedge (d\omega)^k$ vanishes on a hypersurface.
For the definition of singular contact structures, we relax the nowhere vanishing assumption in the definition of a contact structure.

\begin{definition}
     A collection of pairs $\omega = \{(U_i, \omega_i)\}$, as in the definition of contact structures, is called a {\bf singular contact form} on $M$ if it satisfies
     \begin{enumerate}
    \item $\omega_i \wedge (d\omega_i)^k$ vanishes only on a nowhere dense subset of $U_i$, and
    \item for any $i,j \in I$, there exists $f_{ij} \in \Gamma (U_i \cap U_j, \mathcal{O}_M^*)$ such that $\omega_i = f_{ij} \omega_j$ on $U_i \cap U_j$.
     \end{enumerate}
\end{definition}
A singular contact structure is defined similarly to a contact structure.
Let $S$ be the locus of noncontact points on $M$, that is,
$$
S \cap U_i = \{ x \in M \, : \, (\omega_i \wedge (d\omega_i)^k) (x) = 0\}
$$
for each $i \in I$. We call $S$ the {\bf Martinet hypersurface}. For a local holomorphic $(2k+1)$-form $\Omega_i$ on $U_i$, we can find a holomorphic function $H_i$ on $U_i$ such that
$$
\omega_i \wedge (d\omega_i)^k = H_i \Omega_i.
$$
Note that $H_i$ is independent of the choice of a holomorphic $(2k+1)$-form up to nowhere vanishing holomorphic functions. We say that $S$ is {\bf structurally smooth} at $x \in U_i \subseteq M$ if $dH_i (x) \neq 0$.

\begin{definition}
     A singular contact form $\omega$ is said to {\bf have singularities of the first type} if $\omega$ never vanishes on the Martinet hypersurface (and hence on the whole manifold $M$). Otherwise, we say $\omega$ {\bf has singularities of the second type}.
\end{definition}

As mentioned in the introduction, we focus on singular contact structures of the first type with structurally smooth Martinet hypersurface in this article.
Let us now consider a local situation. For the sake of simplicity, we may assume that $M$ is $\CC^{2k+1}$ with a singular contact form $\omega$ and that
$$
\omega \wedge (d\omega)^k = H dz \wedge dx^1 \wedge \cdots \wedge dx^{2k},
$$
where $(z, x^1,\ldots ,x^{2k})$ is a coordinate system of $\CC^{2k+1}$ and $H(0,\ldots,0) = 0$.

A local Pfaffian equation $\alpha$ on $\CC^{2k}$ is said to be {\bf realizable} ({\bf realizable with structurally smooth $S$}, respectively) if there exists a nowhere vanishing $1$-form $\omega$ on $\CC^{2k+1}$ such that $S = \CC^{2k} := \{0\} \times \CC^{2k}$ and $\alpha = \omega|_{\CC^{2k}}$ (structurally smooth $S = \CC^{2k}$ and $\alpha = \omega|_{\CC^{2k}}$ respectively).

According to \cite[Theorem 5.1]{JZ01}, we have the following characterization of realizability.
\begin{theorem} \label{THM:JZRealizability}
A local Pfaffian equation $\alpha$ is realizable with structurally smooth $S$ if and only if there exists a nonvanishing $1$-form $\beta$ on $\CC^{2k}$ such that
\begin{enumerate}
\item $(d\alpha)^k = k (d\alpha)^{k-1} \wedge \alpha \wedge \beta$ on $S$,
\item $(d\alpha)^{k-1} \wedge d\beta - (k-1)(d\alpha)^{k-2} \wedge \alpha \wedge \beta \wedge d\beta $ does not vanish at the origin.
\end{enumerate}
In this case, the nowhere vanishing $1$-form $\omega$ can be written as $\omega = dz + \alpha + z \beta$. 
\end{theorem}

\begin{remark}
In both cases of nonsingular contact structures and singular contact structures of the first type, normalization theorems such as Darboux's theorem and Theorem~\ref{THM:JZRealizability} play a crucial role in studying the contact Hamiltonian map. Singularities and normalizations of singular contact structures of the first type have been studied extensively in \cite{M70, JZ95, JZ01}. Meanwhile, singular contact structures of the second type have been studied in \cite{Ly75, JTN24}.
\end{remark}

From now on, we assume that $\alpha$ is realizable with structurally smooth $S = \CC^{2k}$.
Denote by $\mathfrak{a}(\CC^{2k+1}, \omega)_0$ the set of germs of infinitesimal contact transformations near the origin. In other words, $\mathfrak{a}(\CC^{2k+1},\omega)_0$ is the set of germs of vector fields $X$ defined near the origin such that
$$
\mathcal{L}_X \omega = h  \omega
$$
for some $h \in \mathcal{O}_{2k+1,0}$. Define the map 
$$
\theta : \mathfrak{a}(\CC^{2k+1},\omega)_0 \longrightarrow \mathcal{O}(L)_0 \simeq \mathcal{O}_{\CC^{2k+1},0}
$$
by $\theta(X) = \omega (X)$ for each $X \in \mathfrak{a}(\CC^{2k+1},\omega)_0$.
Note that one of the purposes of this article is to find a condition for $X$ to be an element of $\mathfrak{a}(\CC^{2k+1},\omega)_0$ and to describe the image of the above map in terms of $\alpha$ and $\beta$. To do this, let us briefly recall the notion of Pfaffians and its relation with exterior products.

Let $W = [\xi_{i,j}]$ be a skew-symmetric $2k \times 2k$ matrix and consider the associated $2$-form $\xi = \sum_{i,j}{\xi_{i,j} dx^i \wedge dx^j}$ on $\CC^{2k}$. For positve integers $1 \le m \le k$, denote by $\mathfrak{S}_{m,k}$ the set of injective maps $\sigma$ from $\{1,\ldots, m\}$ to $\{1,\ldots, k\}$. Then $\xi^2$ can be written as
$$
\xi^2 = \sum_{\sigma \in \mathfrak{S}_{4,2k}} \xi_{\sigma(1),\sigma(2)}\xi_{\sigma(3),\sigma(4)} dx^{\sigma(1)} \wedge dx^{\sigma(2)} \wedge dx^{\sigma(3)} \wedge dx^{\sigma(4)}.
$$
More generally, for $1 \le m \le k$, we have
$$
\xi^m = \sum_{\sigma \in \mathfrak{S}_{2m,2k}} \left( \prod_{i=1}^m \xi_{\sigma(2i-1),\sigma(2i)} \right) dx^{\sigma(1)} \wedge \cdots \wedge dx^{\sigma(2m)}.
$$
In the case when $m=k$, after rearranging $dx^{\sigma(1)} \wedge \cdots \wedge dx^{\sigma(2k)}$ into $dx^1 \wedge \cdots \wedge dx^{2k}$, we have
$$
\xi^k = \left( \sum_{\sigma \in \mathfrak{S}_{2k}} \mathrm{sgn}(\sigma) \prod_{i=1}^{k} \xi_{\sigma(2i-1),\sigma(2i)} \right) dx^1 \wedge \cdots \wedge dx^{2k},
$$
where $\mathfrak{S}_{2k} = \mathfrak{S}_{2k,2k}$ and $\mathrm{sgn}(\sigma)$ is the signature of a permutation $\sigma$. The coefficient of $dx^1 \wedge \cdots \wedge dx^{2k}$ divided by $2^k k!$ in the above is called the {\bf Pfaffian} of a skew-symmetric $2k \times 2k$ matrix $W = [\xi_{i,j}]$. We denote by $\pf(W)$ the Pfaffian of $W$. Since the Pfaffian of a skew-symmetric matrix of order $2k$ is defined via permutations, there is a close relation between the Pfaffian of a skew-symmetric matrix and its determinant. The following are well-known facts in linear algebra.
\begin{proposition}
With notation as above, The Pfaffian of $W$ has the following properties.
\begin{enumerate}
\item The determinant of $W$ is the square of the Pfaffian of $W$, that is, $\det(W) = \pf(W)^2$.
\item There is an expansion formula for $\pf(W)$. For any fixed $i \in \{1,\ldots, 2k\}$,
$$
\pf(W) = \sum_{\substack{j=1 \\ j \neq i}}^{2k} (-1)^{i+j+1+H(i-j)} \xi_{i,j} \pf(W_{\hat{\imath}\hat{\jmath}}^{\hat{\imath}\hat{\jmath}}).
$$
Here, $H(t)$ is the Heavisde function and $W_{\hat{\imath}\hat{\jmath}}^{\hat{l}\hat{m}}$ is the $(2k-2) \times (2k-2)$ matrix obtained from removing the $i$-th and $j$-th rows and the $l$-th and $m$-th columns. Note that $W_{\hat{\imath}\hat{\jmath}}^{\hat{\imath}\hat{\jmath}}$ is a skew-symmetric $(2k-2) \times (2k-2)$ matrix.
\item The cofactor $(-1)^{i+j} \det( W_{\hat{\imath}}^{\hat{\jmath}})$ with respect to $(i,j)$ can be obtained from the Pfaffian cofactor with respect to $(i,j)$, that is,
$$
(-1)^{i+j} \det( W_{\hat{\imath}}^{\hat{\jmath}}) = (-1)^{i+j+1+H(i-j)}\pf(W_{\hat{\imath}\hat{\jmath}}^{\hat{\imath}\hat{\jmath}}) \pf(W).
$$
\end{enumerate}
\end{proposition}

\bigskip

\section{Main statements} \label{SEC:MainStatements}

In this section, we present our main results: the theorem mentioned in the introduction and its local version with explicit formula.

\begin{theorem} \label{THM:EvaluatingIsInjective}
     Let $\omega = \{\omega_i\}$ be a singular contact structure of the first type with structurally smooth Martinet hypersurface and $L$ the quotient bundle $T^{1,0}M/\mathrm{Ker}(\omega)$. Then the the sequence of sheaves
     $$
     0 \longrightarrow \mathfrak{a}(M,\widetilde \omega) \stackrel{\theta}{\longrightarrow} \mathcal{O}(L)
     $$
     is exact.
\end{theorem}

To prove the above theorem, it suffices to prove that it is injective locally, as in the nonsingular case. Let $(z, x^1,\ldots, x^{2k})$ be a local coordinate system. The germ of a singular contact structure $\omega$ of the first kind with structurally smooth Martinet hypersurface can be written as
$$
\omega = dz + \alpha + z\beta,
$$
where $\alpha = \sum_{i=1}^{2k} a_i(x) dx^i$ and $\beta = \sum_{i=1}^{2k} b_i(z,x) dx^i$ are $1$-forms as in Theorem~\ref{THM:JZRealizability}. Being a nowhere vanishing $1$-form, the singular contact form $\omega$ decomposes $T^{1,0}\CC^{2k+1}$ into two parts, the line generated by $Z := \partial / \partial z$ and the kernel of $\omega$. With this decomposition, write a holomorphic vector field $X$ as
$$
X = fZ + \sum_{i=1}^{2k}s_i \left(\pp{x^i} - (a_i + zb_i)Z \right) \in \langle Z \rangle \oplus \ker (\omega)_0.
$$
In this setting, Theorem~\ref{THM:EvaluatingIsInjective} immediately follows from the following two results.
\begin{lemma} \label{LEM:InfAutoCondition}
    With notation as above, the germ of a vector field $X$ at the origin is an infinitesimal contact transformation if and only if
    $$
        \imath_Z \imath_X (\omega \wedge d\omega) = \imath_Z ( f d\omega + \omega \wedge df).
    $$
    Furthermore, in this case, the Lie derivative of $\omega$ with respect to $X$ is given by
    $$
        \mathcal{L}_X \omega = \left(  \pp[f]{z} -\sum_{m=1}^{2k}{ s_m b_m  } {- z\sum_{m=1}^{2k}{ s_m \pp[b_m]{z} } } \right) \omega.
    $$
\end{lemma}

\begin{theorem} \label{THM:InfAutoFromFtn}
     With notation as above, the map $\theta:\mathfrak{a}(\CC^{2k+1},\omega)_0 \to \mathcal{O}_{\CC^{2k+1},0}$ is injective. The germ of a holomorphic function $f \in \mathcal{O}_{\CC^{2k+1},0}$ is in the image of the map $\theta$ if and only if the form
    $$
        dz \wedge \left( df + f \beta - \pp[f]{z} \alpha \right) \wedge (d\alpha - \alpha \wedge \beta)^{k-1}
    $$
    vanishes identically on $S = \CC^{2k}$. In this case, $\theta^{-1}(f)$ is can be written as
    $$
        \theta^{-1}(f) = f Z + \sum_{m=1}^{2k}{s_m \left[ \pp{x^m} - (a_m + zb_m) Z \right]},
    $$
    where $s_m$ is given by
    $$
        s_m = k \cdot \frac{dx^m \wedge (f d\omega + \omega \wedge df) \wedge (d\alpha + zd\beta - \alpha \wedge \imath_Z d (z\beta))^{k-1}}{dz \wedge (d\alpha + zd\beta - \alpha \wedge \imath_Z d (z\beta))^k}.
    $$
\end{theorem}

\bigskip

\section{Proof of Lemma~\ref{LEM:InfAutoCondition}} \label{SEC:ProofOfLEM:InfAutoCondition}

Let $X$ be a holomorphic vector field defined near the origin. Then, $X$ is an infinitesimal contact transformation if and only if there exists $h \in \mathcal{O}_{2k+1,0}$ satisfying
\begin{equation} \label{EQN:InfAuto}
\imath_X ( d\omega ) + d(\omega(X)) = h \omega.
\end{equation}
Now let us decompose $X$ as
$$
X = f Z  + X' \in \left\langle Z \right\rangle \bigoplus \mathrm{Ker}(\omega)_0 = \mathcal{O}(T\CC^{2k+1})_0,
$$
where $Z = \p{z}$. Observe that $X' \in \mathrm{Ker}(\omega)$ can be written as
$$
X' = \sum_{i=1}^{2k}{ s_i \left( \pp{x_i} - (a_i + zb_i) Z \right)}
$$
for some holomorphic functions $s_1,\ldots, s_{2k}$. Denoting the vector appearing in the parentheses of the $i$-th summand by $X_i$, we obtain a local frame $\{X_1, \ldots, X_{2k} \}$ for $\mathrm{Ker}(\omega)_0$.
Note that the dual frame of $\{Z, X_1, \ldots, X_{2k}\}$ is $\{\omega, dx^1, \ldots, dx^{2k} \}$.
 
On the other hand, since $f = \omega(X)$, the equation (\ref{EQN:InfAuto}) becomes
\begin{equation} \label{EQN:InfAutoWithFunction}
\imath_X (d\omega) + df = h \omega.
\end{equation}
The $1$-form $\imath_X(d\omega)$, being an element of $\mathcal{O}((T^{1,0}\CC^{2k+1})^*)_0$, can be written as a linear combination of $dx^1, \ldots, dx^{2k}$ and $\omega$. For this purpose, we are going to calculate $\imath_X(d\omega)(Z)$ and $\imath_X(d\omega)(X_l)$'s explicitly.
\begin{align*}
\imath_X (d\omega)(Z) = d\omega(X',Z) &= d\alpha (X',Z) + (dz\wedge\beta) (X',Z) + zd\beta(X',Z) \\
&= -\beta(X') - z(\partial_z \beta)(X'),
\end{align*}
where $\partial_z \beta$ denotes the form of the same degree with $\beta$ obtained by differentiating all the coefficients of $\beta$ with respect to $z$. Later, we will replace $\partial_z$ by $\imath_Z d$ when it acts on forms not involving $dz$ and on holomorphic functions.
For the expression of $\imath_X (d\omega)(X_l)$, we shall use explicit forms of $d\alpha$ and $d\beta$ as in the case of $\imath_X (d\omega)(Z)$:
\begin{align*}
d \alpha &= \sum_{i,j}{ \pp[a_j]{x^i} dx^i \wedge dx^j }, \\
d \beta &= \sum_{i,j}{ \pp[b_j]{x^i} dx^i \wedge dx^j } + dz \wedge \partial_z \beta.
\end{align*}
Since $\imath_X (d\omega)(X_l) = d\alpha(X,X_l) + (dz\wedge\beta)(X,X_l) + zd\beta(X,X_l)$, each term in the right-hand side is given as follows:
\begin{align*}
d\alpha(X,X_l) =& \sum_{m=1}^{2k}{s_m \left( \pp[a_l]{x^m} - \pp[a_m]{x^l} \right) },\\
d \beta(X, X_l) =& \sum_{m=1}^{2k}{ s_m \left( \pp[b_l]{x^m} - \pp[b_m]{x^l} \right)}
+ \pp[b_l]{z}\left(f - \sum_{i=1}^{2k}{ s_i (a_i+zb_i) } \right) \\ 
& + (a_l + zb_l)\sum_{i=1}^{2k}{s_i \pp[b_i]{z} }, \\ 
(dz \wedge\beta )(X,X_l) =& b_l \left( f-\sum_{i=1}^{2k}{ s_i(a_i+zb_i) } \right) + (a_l + zb_l) \sum_{i=1}^{2k}{ s_i b_i }.
\end{align*}
From these equalities, we know that the coefficient of $dx^l$ in $\imath_X (d\omega) + df$ with respect to the dual frame $\{\omega, dx^1, \ldots, dx^{2k}\}$ is equal to
\begin{equation} \label{EXP:TempCoeffOfdx_l}
\begin{aligned}
&\sum_{m=1}^{2k}{ s_m \left\{ \pp[(a_l + zb_l)]{x^m} - \pp[(a_m + zb_m)]{x^l} + b_m (a_l + zb_l) - b_l (a_m + zb_m) \right\} } \\
& + b_l f + z\pp[b_l]{z}\left(f - \sum_{m=1}^{2k}{ s_m (a_m+zb_m) } \right) \\ 
& + z(a_l + zb_l)\sum_{m=1}^{2k}{s_m \pp[b_m]{z} } + \pp[f]{x^l} - (a_l + zb_l)\pp[f]{z},
\end{aligned}
\end{equation}
\noindent which must be zero by (\ref{EQN:InfAutoWithFunction}).
On the other hand, the coefficient of $\omega$ in $\imath_X (d\omega) + df$, which is given by
$$
\pp[f]{z} -\sum_{m=1}^{2k}{ s_m b_m  } - z\sum_{m=1}^{2k}{ s_m \pp[b_m]{z} },
$$
should equal to $h$. This shows that a germ of a vector field $X  = f Z + \sum{ s_j X_j }$ is an infinitesimal contact transformation of $\omega$ if and only if $f$ and $s_j$'s satisfies the equation obtained by letting (\ref{EXP:TempCoeffOfdx_l}) be zero for all $l$. In other words, $X$ is an infinitesimal contact transformation of $\omega$ if and only if 
\begin{equation} \label{EQN:InfAutoConditionExpanded}
\sum_{m=1}^{2k}{ s_m \left( \pp[\eta_l]{x^m} - \pp[\eta_m]{x^l} + \eta_l \pp[\eta_m]{z} - \eta_m \pp[\eta_l]{z} \right) }
= \eta_l \pp[f]{z} - f \pp[\eta_l]{z} - \pp[f]{x^l}
\end{equation}
holds for all $l = 1,\ldots 2k$, where $\eta_m = a_m + zb_m$. These $2k$ equations in (\ref{EQN:InfAutoConditionExpanded}) can be expressed by a single equation in terms of differential forms
\begin{equation} \label{EQN:InfAutoCondition} 
\imath_Z \imath_X ( \omega \wedge d\omega ) = \imath_Z (f d\omega + \omega \wedge df).
\end{equation}
This complete the proof of Lemma~\ref{LEM:InfAutoCondition}.

\bigskip

\section{Proof of Theorem~\ref{THM:InfAutoFromFtn}} \label{SEC:ProofOfTHM:InfAutoFromFtn}

The injectivity of $\theta$ follows immediately from the formula for $\theta^{-1}(f)$ since it will be determined uniquely. To work with (\ref{EQN:InfAutoConditionExpanded}), we use Pfaffians of matrices introduced in the preliminaries.

Let us consider the condition (\ref{EQN:InfAutoConditionExpanded}) for infinitesimal contact transformations. For each $j=1,\ldots, 2k$, denote by $\mathbf{w}_j$ the $1 \times 2k$ matrix whose $i$-th row is 
\begin{equation} \label{EXP:EntryOfW}
\xi_{i,j} := \pp[\eta_i]{x^j} - \pp[\eta_j]{x^i} + \eta_i \pp[\eta_j]{z} - \eta_j \pp[\eta_i]{z}
\end{equation}
and let $W$ be the skew-symmetric $2k \times 2k$ matrix whose $j$-th column is $\mathbf{w}_j$. Denote by $\mathbf{d}$ the column vector in $\CC^{2k}$ whose $i$-th component is 
$$
d_l := \eta_l \pp[f]{z} - f \pp[\eta_l]{z} - \pp[f]{x^l}.
$$
It is worth noting that $\xi_{i,j}$ is equal to the coefficient of $dx^i \wedge dx^j$ in $-2(d\alpha + zd\beta - \alpha \wedge \partial_z (z\beta))$. 

In these setting, each $s_m$ in (\ref{EQN:InfAutoConditionExpanded}) can be found {\it formally} by Cramer's rule:
$$
s_m = \frac{\det [\mathbf{w}_1 \enskip \cdots \enskip \mathbf{w}_{m-1} \enskip \mathbf{d} \enskip \mathbf{w}_{m+1} \enskip \cdots \enskip \mathbf{w}_{2k}]}{\det{W}}.
$$
We can say that the above solutions really exist at least as meromorphic functions only when $\det{W}$ is not identically vanishing near the origin. Moreover, the desired holomorphic solutions exist only when $s_m \in \mathcal{O}(\CC^{2k+1})_0$. Equivalently, the vanishing order of the numerator in $s_m$ must be not less than that of the denominator. 

For this, observe first that 
$$
\omega \wedge (d\omega)^k = dz \wedge (d\alpha + zd\beta - \alpha \wedge \partial_z (z\beta))^k + \mathrm{o}(z),
$$
which implies that the $1$-jet of $\pf(W)$ with respect to $z$ coincides with that of $\omega \wedge (d\omega)^k / (dz\wedge dx^1 \wedge \cdots \wedge dx^{2k})$. The realizability condition and the structural smoothness condition in Theorem~\ref{THM:JZRealizability} guarentee that $\pf(W)$ is of the form $cz + \mathrm{o}(z)$ with $c \in \mathcal{O}(\CC^{2k})_0$ and $c(0) \neq 0$. Thus, we have $\det(W) = c^2 z^2 + \mathrm{o}(z^2)$, which concludes that the vanishing order of $\det(W)$ along the Martinet surface is $2$. Moreover, this shows that an infinitesimal contact transformation for which the coefficient of $Z$ is $f$ can always be found at least as a {\it meromorphic vector field}.

Using the cofactor expansion for $\det [\mathbf{w}_1 \enskip \cdots \enskip \mathbf{w}_{m-1} \enskip \mathbf{d} \enskip \mathbf{w}_{m+1} \enskip \cdots \enskip \mathbf{w}_{2k}]$ with respect to the $m$-th column, we have
\begin{align*}
\det [\mathbf{w}_1 \, \cdots \, \mathbf{w}_{m-1} \, \mathbf{d} \, \mathbf{w}_{m+1} \, \cdots \, \mathbf{w}_{2k}] &= \sum_{j=1}^{2k} (-1)^{m+j} d_j \det (W_{\hat{\jmath}}^{\hat{m}}) \\
&= \pf(W) \sum_{j=1}^{2k} (-1)^{m+j+1+H(j-m)} d_j \pf(W_{\hat{\jmath}\hat{m}}^{\hat{\jmath}\hat{m}}).
\end{align*}
So we can conclude that $s_m \in \mathcal{O}(\CC^{2k+1})_0$ if and only if 
$$
\sum_{\substack{j=1 \\ j \neq m}}^{2k} (-1)^{m+j+1+H(j-m)} d_j \pf(W_{\hat{\jmath}\hat{m}}^{\hat{\jmath}\hat{m}}) \in (z) \subset \mathcal{O}(\CC^{2k+1})_0.
$$
To express $s_m$ in terms of differential forms, consider the following equalities:
\begin{multline*}
(d\alpha + z d\beta - \alpha \wedge \partial_z (z\beta))^{k-1} \\
=  \sum\nolimits' { (-1)^{k-1}(k-1)!\pf(W_{\hat{\jmath}\hat{m}}^{\hat{\jmath}\hat{m}}) dx^1\wedge \cdots \wedge \widehat{dx^j} \wedge \cdots \wedge \widehat{dx^m} \wedge \cdots \wedge dx^{2k} },
\end{multline*}
and
$$
(fd\omega + \omega \wedge df) = \sum_{j=1}^{2k} d_j dz \wedge dx^j + (\text{terms not involving $dz$}).
$$
Here, $\sum\nolimits'$ means that the summation is taken so that $dx^1,\ldots,dx^{2k}$ are arranged in increasing order. From these equalities, we have
\begin{multline*}
dx^m \wedge (fd\omega + \omega \wedge df) \wedge (d\alpha + zd\beta - \alpha \wedge \partial_z (z \beta))^{k-1} \\
= \left( \sum_{j \neq m} (-1)^{k + m + j + 1 + H(j-m)} (k-1)! d_j \pf(W_{\hat{\jmath}\hat{m}}^{\hat{\jmath}\hat{m}}) \right) dz \wedge dx^1 \wedge \cdots \wedge dx^{2k}.
\end{multline*}
On the other hand, by the definition of the Pfaffian of $W$, we have
$$
dz \wedge (d\alpha + zd\beta - \alpha \wedge \partial_z (z\beta))^k = (-1)^k k! \pf(W) dz \wedge dx^1 \wedge \cdots \wedge dx^{2k}.
$$
We conclude that
\begin{equation} \label{EXP:FormulaFor:s_m}
s_m = k \cdot \frac{dx^m \wedge (f d\omega + \omega \wedge df) \wedge (d\alpha + zd\beta - \alpha \wedge \partial_z (z\beta))^{k-1}}{dz \wedge (d\alpha + zd\beta - \alpha \wedge \partial_z (z\beta))^k}
\end{equation}
and that $f \in \theta(\mathfrak{a}(\CC^{2k+1},\omega)_0)$ if and only if 
\begin{equation} \label{COND:FunctionInTheImage}
\frac{dx^m \wedge ( f d\omega + \omega \wedge df ) \wedge (d\alpha - \alpha \wedge \beta)^{k-1}}{dz \wedge dx^1 \wedge \cdots \wedge dx^{2k}} \in (z) \subset \mathcal{O}(\CC^{2k+1})_0
\end{equation}
for all $m=1,\ldots, 2k$. The above $2k$ conditions can be expressed by one differential form since the coefficient of $dz \wedge dx^1\wedge\cdots\wedge dx^{2k}$ in $dx^m \wedge (f d\omega + \omega \wedge df) \wedge (d\alpha - \alpha \wedge \beta)^{k-1}$ vanishes along $S$ if and only if the coefficient of $dx^1 \wedge \cdots \wedge \widehat{dx^m} \wedge \cdots \wedge dx^{2k}$ in
\begin{equation}  \label{COND:FunctionInTheImageInSingleEquation}
\left( d'f + f \beta - \pp[f]{z} \alpha \right) \wedge (d\alpha - \alpha \wedge \beta)^{k-1}
\end{equation}
vanishes along $S$, where $d'$ is the exterior derivative on $\CC^{2k} = \CC^{2k+1} \cap \{ z=0\}$. Thus (\ref{COND:FunctionInTheImage}) holds if and only if there exists a differential form $\gamma$ such that the form in (\ref{COND:FunctionInTheImageInSingleEquation}) is equal to $z \gamma$, that is,
\begin{equation} \label{EQN:FunctionInTheImageWithGammaInTheProof}
\left( d'f + f \beta - \pp[f]{z} \alpha \right) \wedge (d\alpha - \alpha \wedge \beta)^{k-1} = z\gamma.
\end{equation}
As we have observed in the case when $k = 1$, we cannot deduce that $f$ is contained in the ideal $(z)$ even though (\ref{COND:FunctionInTheImage}) holds for all $m =1,\ldots, 2k$. Now, assume that $f$ satisfies (\ref{COND:FunctionInTheImage}) for all $m=1,\ldots, 2k$. Then the corresponding vector field $X \in \mathfrak{a}(\CC^{2k+1},\omega)_0$ can be written as
$$
X = f Z + \sum_{m=1}^{2k}{s_m \left[\pp{x^m} - (a_m+zb_m)Z \right]}.
$$
The injectivity of $\theta$ follows immediately as we mentioned at the beginning of the proof.

\begin{remark}
In the proofs of Lemma~\ref{LEM:InfAutoCondition} and Theorem~\ref{THM:InfAutoFromFtn}, none of the properties of holomorphic functions have been used except their analyticity. Therefore, all results presented in this paper also hold on the category of real analytic manifolds as well, provided that the structurally smooth Martinet hypersurface is locally $\RR^{2k}$ in $\RR^{2k+1}$.
\end{remark}

\bigskip

\section{Applications} \label{SEC:Applications}

As Theorem~\ref{THM:InfAutoFromFtn} provides an explicit local formula for an infinitesimal contact transformation $X$ whose image under $\theta$ is $f$, we obatin the following immediate consequence.
\begin{corollary} \label{CORO:VanishingOrder1ImpliesVanishingOrder2}
Let $M$ be a complex manifold of complex dimension $2k+1$ ($k \ge 1$), and let $\omega$ be a singular contact structure on $M$ of the first type with structurally smooth Martinet hypersurface $S$. If $L = T^{1,0}M/\mathrm{Ker}(\omega)$ and $I_S$ is the ideal sheaf of $S$, then we have
$$
\theta(\mfa(M,\omega)) \cap H^0(M, L \otimes I_S) \subseteq H^0(M, L \otimes I_S^2).
$$
In particular, the infinitesimal contact transformation $X$ corresponding to $f \in \theta(\mfa(M,\omega)) \cap H^0(M, L \otimes I_S)$ vanishes everywhere on $S$.
\end{corollary}
\begin{proof}
This immediately follows from the condition for $f \in \theta(\mfa(M,\omega)) \cap H^0(M, L \otimes I_S)$ and the formula for $\theta^{-1}(f)$ in Theorem~\ref{THM:InfAutoFromFtn}.
\end{proof}

In the two simple examples below, by finding the infinitesimal contact transformation $X$ for a given function in the image of $\theta$, we observe two interesting properties.

\begin{example} \label{EXAM:SimpleCase1}
Consider $\alpha = e^{x^1 x^2} dx^1$ and $\beta = -x^1 dx^2$ in $\CC^3$ with coordinates $(z,x^1,x^2)$. Then $\omega = dz + \alpha + z \beta$ is a singular contact form with structurally smooth Martinet hypersurface $\CC^2 = (z = 0)$. Solving (\ref{EQN:FunctionInTheImageWithGammaInTheProof}) with $k=1$ and $f = \sum{c_j(x) z^j}$, one can easily check that $f \in \theta(\mfa(\CC^3,\omega))$ if and only if there exists a holomorphic function $g(x^1)$ such that
$$
f = e^{x^1x^2} g(x^1) + (x^2 g(x^1) + g'(x^1))z + \mathrm{o}(z).
$$
\end{example}

\begin{example} \label{EXAM:SimpleCase2}
Consider $\alpha = (x^2+1)^{-1}e^{-x^1x^2} dx^1$ and $\beta = (x^1 + (x^2+1)^{-1}) dx^2$ near the origin in $\CC^3$ with coordinates $(z, x^1, x^2)$. Then $\omega = dz + \alpha + z\beta$ is realizable with structurally smooth $S$ near the origin. 
As in the first example, one can show that $f \in \theta( \mathfrak{a}(\CC^3, \omega))$ if and only if there exists a holomorphic function $g(x^1)$ defined near the origin such that
$$
f = (x^2 + 1)^{-1}e^{-x^1x^2}g(x^1) + (g'(x^1) - x^2g(x^1))z + o(z).
$$
For simplicity, suppose that $g(x^1) \equiv 1$ and there are no higher-order terms in $z$, that is,
$$
f = (x^2+1)^{-1} e^{-x^1x^2} -x^2z.
$$
\end{example}

In the above examples, note that $f$ vanishes identically along $S$ if and only if $g(x^1) \equiv 0$. When this is the case, the vanishing order of $f$ along $S$ is at least $2$ as we have seen in Corollary~\ref{CORO:VanishingOrder1ImpliesVanishingOrder2}. Moreover, a direct computation shows that the coefficient of $Z = \p{z}$ of the corresponding infinitesimal contact transformation $X$ vanishes along $S$ even if $f$ does not vanish along $S$.

If $X$ is an infinitesimal contact transformation of a singular contact structure $\omega$, it is naturally expected that the $1$-parameter group of contact transformations generated by $X$ keeps any noncontact points staying on the locus of noncontact points. In this point of view, every infinitesimal contact transformation must be tangent to the Martinet hypersurface as shown in the following theorem.

\begin{theorem} \label{THM:PreserveMartinetHypersurface}
Let $M$ be a complex manifold of complex dimension $2k+1$ ($k \ge 1$), and let $\omega$ be a singular contact structure on $M$ of the first type with structurally smooth Martinet hypersurface $S$. If a holomorphic vector field $X$ is an infinitesimal contact transformation of $\omega$, then the restriction $X|_S$ to the Martinet hypersurface $S$ is a holomorphic tangent vector field on $S$.
\end{theorem}

\begin{proof}
It is enough to show that this theorem is true locally. Consider the germ of the singular contact structure $\omega$ at $p \in S$. We may assume that $\omega$, $S$ and $X$ are given as in Section~\ref{SEC:MainStatements}. We will show that the coefficient of $Z$ of $X$ vanishes along $S$. Since $X$ is an infinitesimal contact transformation of $\omega$, there is a differential form $\gamma$ such that
\begin{equation} \label{EQN:FunctionInTheImageWithGamma}
\left( d'f + f \beta - \pp[f]{z} \alpha \right) \wedge (d\alpha - \alpha \wedge \beta)^{k-1} = z\gamma.
\end{equation}
This characterization will play a crucial role in the proof. We will conclude the theorem by showing that the constant term of the series expansion with respect to $z$ of the function 
\begin{align*}
&f - \sum_{m=1}^{2k} (a_m + z b_m) s_m \\
&\quad \quad = f - k \cdot \frac{(\alpha + z \beta) \wedge (f d\omega + \omega \wedge df) \wedge (d\alpha + zd\beta - \alpha \wedge \partial_z (z\beta))^{k-1}}{dz \wedge (d\alpha + zd\beta - \alpha \wedge \partial_z (z\beta))^k}
\end{align*}
at the origin is zero. Equivalently, we will demonstrate that the $1$-jet of the form
\begin{equation} \label{EQN:1JetOfzPartOfX}
\begin{aligned}
&f dz \wedge (d\alpha + zd\beta - \alpha \wedge \partial_z (z\beta))^k \\ 
&\quad - k(\alpha + z\beta) \wedge (f d\omega + \omega \wedge df) \wedge (d\alpha + zd\beta - \alpha \wedge \partial_z (z\beta))^{k-1}
\end{aligned}
\end{equation}
vanishes. This suffices to show that (\ref{EQN:1JetOfzPartOfX}) is equal to zero modulo $z^2$. Expanding $\omega$ and $d\omega$ to differential forms written in $\alpha$, $\beta$ and $dz$, we obtain the following equality. In the rest of the proof, all equalities are understood modulo $z^2$.
\begin{align*}
&f dz \wedge (d\alpha + zd\beta - \alpha \wedge \partial_z (z\beta))^k \\ 
&\quad - k(\alpha + z\beta) \wedge (fd\omega + \omega \wedge df) \wedge (d\alpha + z d\beta - \alpha \wedge \partial_z (z\beta))^{k-1} \\
&= dz \wedge \{f(d\alpha)^k + kzf(d\alpha)^{k-1} \wedge d\beta \\ 
&\quad + k((d\alpha)^{k-1} \wedge (\alpha + z\beta) + (k-1)z(d\alpha)^{k-2}\wedge d\beta  \wedge \alpha)  \wedge d'f \}.
\end{align*}
Now replacing $(d\alpha)^{k-1}\wedge d'f$ by the other terms in (\ref{EQN:FunctionInTheImageWithGamma}) gives
\begin{align*}
&f dz \wedge (d\alpha + zd\beta - \alpha \wedge \partial_z (z\beta))^k \\ 
& \quad - k(\alpha + z\beta) \wedge (fd\omega + \omega \wedge df) \wedge (d\alpha + z d\beta - \alpha \wedge \partial_z (z\beta))^{k-1} \\
&=  dz \wedge \left(
\begin{aligned}
&f((d\alpha)^k - k(d\alpha)^{k-1}\wedge \alpha \wedge \beta) + kzf(d\alpha)^{k-1} \wedge d\beta \\
& - kz \pp[f]{z} (d\alpha)^{k-1} \wedge \alpha \wedge \beta + k(k-1)z(d\alpha)^{k-2} \wedge d\beta \wedge \alpha \wedge d'f \\ 
& -kz\gamma \wedge \alpha 
\end{aligned}
\right).
\end{align*}
Taking $d'$ on both sides of (\ref{EQN:FunctionInTheImageWithGamma}) and replacing 
$$
kzf(d\alpha)^{k-1}\wedge d\beta + k(k-1)z (d\alpha)^{k-2} \wedge d\beta \wedge \alpha \wedge d'f
$$
by the rest of the resulting equality, we obtain
\begin{align*}
&f dz \wedge (d\alpha + zd\beta - \alpha \wedge \partial_z (z\beta))^k \\ 
& \quad - k(\alpha + z\beta) \wedge (fd\omega + \omega \wedge df) \wedge (d\alpha + z d\beta - \alpha \wedge \partial_z (z\beta))^{k-1}\\
&=  dz \wedge \left(
\begin{aligned}
&f((d\alpha)^k - k(d\alpha)^{k-1}\wedge \alpha \wedge \beta) - kz \pp[f]{z} (d\alpha)^{k-1} \wedge \alpha \wedge \beta \\ 
&-kz\gamma \wedge \alpha - kz(d\alpha)^{k-1} \wedge \alpha \wedge d'\left(\pp[f]{z} \right) \\
&+ kz\pp[f]{z} (d\alpha)^k + k^2z(d\alpha)^{k-1}\wedge\beta\wedge d'f
\end{aligned}
\right)
\end{align*}
From the relation (\ref{EQN:FunctionInTheImageWithGamma}) multiplied by $\beta$ on both sides, we have
\begin{align*}
&f dz \wedge (d\alpha + zd\beta - \alpha \wedge \partial_z (z\beta))^k \\ 
& \quad - k(\alpha + z\beta) \wedge (fd\omega + \omega \wedge df) \wedge (d\alpha + z d\beta - \alpha \wedge \partial_z (z\beta))^{k-1}\\
&=  dz \wedge \left(
\begin{aligned}
&f((d\alpha)^k - k(d\alpha)^{k-1}\wedge \alpha \wedge \beta) \\
& + kz\alpha \wedge \gamma  - kz \pp[f]{z} (d\alpha)^{k-1} \wedge \alpha \wedge \beta  - kz(d\alpha)^{k-1} \wedge \alpha \wedge d'\left(\pp[f]{z} \right)
\end{aligned}
\right)
\end{align*}
Modifying the last three terms in the above shows
\begin{align*}
& kz\alpha \wedge \gamma  - kz \pp[f]{z} (d\alpha)^{k-1} \wedge \alpha \wedge \beta  - kz(d\alpha)^{k-1} \wedge \alpha \wedge d'\left(\pp[f]{z} \right) \\
&= kzf (d\alpha)^{k-1}\wedge \alpha \wedge (\partial_z \beta).
\end{align*}
Substituting the above equality into (\ref{EQN:1JetOfzPartOfX}), we have
\begin{align*}
&f dz \wedge (d\alpha + zd\beta - \alpha \wedge \partial_z (z\beta))^k \\ 
&\quad - k(\alpha + z\beta) \wedge (fd\omega + \omega \wedge df) \wedge (d\alpha + z d\beta - \alpha \wedge \partial_z (z\beta))^{k-1}\\
&= f dz \wedge \{ (d\alpha)^k - k(d\alpha)^{k-1}\wedge \alpha \wedge \beta + kz (d\alpha)^{k-1} \wedge \alpha \wedge (\partial_z \beta) \}.
\end{align*}
Write $\beta = \beta_0 + z \beta_1 + \mathrm{o}(z)$ where $\beta_0$ and $\beta_1$ are $1$-forms independent of $z$. Substituting this into the above equality, we finally obtain
\begin{align*}
&f dz \wedge (d\alpha + zd\beta - \alpha \wedge \partial_z (z\beta))^k \\ 
& \quad - k(\alpha + z\beta) \wedge (fd\omega + \omega \wedge df) \wedge (d\alpha + z d\beta - \alpha \wedge \partial_z (z\beta))^{k-1} = 0.
\end{align*}
This shows the coefficient of $Z$ of $X$ is zero on $S$ as we desired.
\end{proof}

The other observation is that the infinitesimal contact transformations $X$ in Example~\ref{EXAM:SimpleCase1} and Example~\ref{EXAM:SimpleCase2} seem to play the role of infinitesimal automorphisms for $\alpha$, that is, $\mathcal{L}_{X|_S} \alpha = h\alpha$ for a (local) holomorphic function $h$ on $S$.

\begin{corollary}
With the same assumption as in Theorem~\ref{THM:PreserveMartinetHypersurface}, let $\alpha$ be the restriction of $\omega$ to $S$. Then, for each point $p$ on $S$, there exists an open neighborhood $U$ of $p$ in $S$ and a holomorphic function $h \in \mathcal{O}_{S}(U)$ such that
$$
\mathcal{L}_{X|_S} \alpha = h \alpha.
$$
Moreover, if we write $\omega = dz + \alpha + z\beta$ locally and $X$ is an infinitesimal contact transformation with $\theta(X) = f$, then
$$
\mathcal{L}_{X|_S} \alpha = \left( \pp[f]{z} - \sum_{m=1}^{2k}s_mb_m \right) \alpha,
$$
where $s_m$'s are given as in Theorem~\ref{THM:InfAutoFromFtn} and restricted to $S = \CC^{2k}$.
\end{corollary}
\begin{proof}
By Theorem~\ref{THM:PreserveMartinetHypersurface}, the Lie derivative with respect to $X$ commutes with the restriction to $S$. In other words, we have
$$
\mathcal{L}_{X|_S} \alpha = (\mathcal{L}_X \omega)|_S.
$$
The corollary immediately follows from this identity.
\end{proof}

\bigskip

\bibliography{Bib.bib}

\quad

{\parindent0pt
\normalsize

\textsc{Hoseob Seo}

Research Institute of Mathematics, Seoul National University

1 Gwanak-ro, Gwanak-gu, Seoul, 08826, Republic of Korea

e-mail: hoseobseo@snu.ac.kr; hsseo.math@gmail.com
}

\end{document}